\documentclass[oneside,11pt]{amsart}
\usepackage{amsmath,amsfonts, latexsym,amssymb, times}
\usepackage[active]{srcltx}
\newtheorem{theorem}{Theorem}[section]

\newtheorem{corollary}[theorem]{Corollary}
\newtheorem{proposition}[theorem]{Proposition}
\theoremstyle{definition}

\numberwithin{equation}{section}

 \makeatletter
\renewcommand*\subjclass[2][2000]{%
  \def\@subjclass{#2}%
  \@ifundefined{subjclassname@#1}{%
    \ClassWarning{\@classname}{Unknown edition (#1) of Mathematics
      Subject Classification; using '1991'.}%
  }{%
    \@xp\let\@xp\subjclassname\csname subjclassname@#1\endcsname
  }%
}
 \makeatother
\newcounter{minutes}
\divide\time by 60
\newcounter{hours}
\multiply\time by 60 \addtocounter{minutes}{-\time}

\begin{document}

\title[Superharmonicity of Logarithm of Jacobian of harmonic mappings]{Superharmonicity of Logarithm of Jacobian of harmonic mappings between surfaces }
\author{David Kalaj}
\address{Faculty of Natural Sciences and Mathematics, University of
Montenegro, Cetinjski put b.b. 81000 Podgorica, Montenegro}
\email{davidk@ac.me}


\def\thefootnote{}
\footnotetext{ \texttt{\tiny File:~\jobname.tex,
          printed: \number\year-\number\month-\number\day,
          \thehours.\ifnum\theminutes<10{0}\fi\theminutes }
} \makeatletter\def\thefootnote{\@arabic\c@footnote}\makeatother

\begin{abstract}
We prove that  the logarithm of the Jacobian of a sense preserving harmonic mapping between surfaces is superharmonic, provided that the Gaussian curvature of the image domain is non-negative.
\end{abstract}
\maketitle
\section{Introduction}

\subsection{Harmonic mappings between Riemann surfaces} Let $(\mathcal{N},\sigma)$ and $(\mathcal{M},\rho)$ be
Riemann surfaces with smooth conformal metrics $\sigma$ and $\rho$, respectively. If
a mapping $f:(\mathcal{N},\sigma)\to(\mathcal{M},\rho)$ is $C^2$, then $f$ is said to be
harmonic (or $\rho$-harmonic) if
\begin{equation}\label{el}
f_{z\overline z}+{(\log \rho^2)}_w\circ f\cdot f_z\,f_{\bar z}=0,
\end{equation}
where $z$ and $w$ are the local parameters on $\mathcal{M}$ and $\mathcal{N}$
respectively (see \cite{jost}). Also $f$ satisfies \eqref{el} if and
only if $f$ is $C^1$ and its Hopf differential
\begin{equation}\label{anal}
\Psi=\rho^2 \circ f \cdot f_z\overline{f_{\bar z}}
\end{equation} is a
holomorphic quadratic differential on $N$. It follows from \eqref{anal}, that (i) every holomorphic mapping is harmonic and (ii) the composition of a harmonic mapping and a holomorphic mapping is harmonic.

From now on, we will assume that $\mathcal{M}$ is parameterized by a complex domain $\Omega$.  Let $\Omega$ be a domain in $\mathbf C$ and $\rho$ be a conformal metric in
$\Omega$. The Gaussian curvature of a double diferentiable metric $\rho$ is given by
$$K_\rho=- \frac{2\Delta\log \rho }{\rho^2 }.$$

 Let $h$ be a solution of  \eqref{el}.

Let $$\partial h= \rho(h(z))h_z,\ \ \ \bar\partial h= \rho(h(z))h_{\bar z},$$ 
 and 
 $$\partial_\sigma h= \frac{\rho(h(z))}{\sigma(z)}h_z, \ \ \ \ \bar\partial_\sigma h= \frac{\rho(h(z))}{\sigma(z)}h_{\bar z},$$and define $$J_h=|\partial h|^2-|\bar\partial h|^2,$$  $$ D_h =|\partial h|^2+|\bar\partial h|^2$$ and $$\Delta f = f_{xx}+f_{yy}.$$
 
 $$J^\sigma_h=|\partial_\sigma h|^2-|\bar\partial_\sigma h|^2$$ and $$\Delta^\sigma f = \frac{1}{\sigma^2}\Delta f.$$ 
Then we have 
\begin{proposition}[Bohner's formulas](See e.g. \cite[p.~113]{jost}). If $h$ is a harmonic mapping then  $$\Delta^\sigma \log |\partial_\sigma h|^2 = K_1-K_2 J^\sigma_h$$ and $$\Delta^\sigma \log |\bar\partial_\sigma h|^2 = K_1+K_2 J^\sigma_h.$$ Here $K_1$ is the Gaussian curvature of $\sigma$ and $K_2$ is the Gaussian curvature of $\rho$.
\end{proposition}

For our approach it will be useful to use the following direct corollary of Bochner' formulas.
\begin{corollary}
If $h$ is a harmonic mapping then  $$\Delta \log |\partial h|^2 = -K_2 J_h$$ and $$\Delta \log |\bar\partial h|^2 = K_2 J_h.$$
\end{corollary}
 
\begin{proposition}\label{prop2} \cite[ p.~10-11]{sy} The functions $|\partial h|$ and $|\bar \partial u|$ are identically zero or they have the isolated zeros with well defined orders.
\end{proposition}

Let us conclude this introduction with the following observation concerning harmonic maps.

For $g:M \mapsto N$ the energy integral is defined by

\begin{equation}\label{harel} E_\rho[g]=\int_{M}
(|\partial_\sigma g|^2+|\bar \partial_\sigma g|^2) dV_\sigma,
\end{equation}
where $\partial_\sigma g$, and $\bar \partial_\sigma g$ are the partial
derivatives taken with respect to the metrics $\varrho$ and
$\sigma$, and $dV_\sigma$ is the volume element on $(M,\sigma)$.
Assume that energy integral of $h$ is bounded. Then $h$ is harmonic
if and only if $h$ is a critical point of the corresponding
functional where the homotopy class of $h$ is the range of this
functional. For this definition and some important properties of
harmonic maps see \cite{sy}. 
\section{Background and statement of the main result}
Manojlovi\'c in \cite{kojic} proved that a  sense preserving Euclidean harmonic mapping between two planar domains has superharmonic logarithm of Jacobian and used this fact to prove the minimum principle for the Jacobian of Euclidean harmonic mappings between planar domains.  This fact has been used previously by the author in \cite[p.~8]{dist}, but here was not treated the superharmonicity of the logarithm of Jacobian. This fact has been re-discovered by Iwaniec and Onninen in \cite{io} and generalized by Iwaniec, Koski and  Onninen in \cite{iko}. The aim of this note is to obtain a similar results for the  class of harmonic mappings between Riemann surfaces with certain curvature conditions.
We prove the following theorem

\begin{theorem}\label{theo}
If $h$ is $\rho$-harmonic mapping between Riemann surfaces $(D,\sigma)$ and $(\Omega, \rho)$ such that the Gauss curvature of $\rho$ is $K_2$ and $J=J_h = \rho^2(h(z))(|h_z|^2-|h_{\bar z}|^2)$ is the jacobian then \begin{equation}\label{del}-\Delta \log J = K_2 D_h + \frac{4\rho^4}{J^2}\left(\frac{|h_z|^2}{|h_{\bar z}|^2}|B|^2+\frac{|h_{\bar z}|^2}{|h_{ z}|^2}|A|^2-2\Re (A\bar B)\right)\end{equation} where $$A =h_{zz} \bar h_{\bar z}+\bar h_{z\bar z} h_z$$

and $$B=h_{z\bar z} \bar h_z+h_{\bar z} \bar h_{zz},$$ and $$D_h = |\partial h|^2+|\bar \partial h|^2.$$  In particular  $-\log J_u$ is subharmonic, if $h$ is a sense-preserving homeomorphism and $K_2\ge 0$.
\end{theorem}
\begin{proof}
In the sequel we make use of the formula
\begin{equation}\label{formula}\Delta \log R=\frac{R\Delta R - |\nabla R|^2}{R^2},\end{equation} or what is the same

\begin{equation}\label{formula1}\Delta R =\frac{ |\nabla R|^2}{R}+R\Delta \log R.\end{equation} Here $\nabla R=(R_x,R_y)$ is the Euclidean gradient of a differentiable non-vanishing real function $R$.
Let $J=J_u$. Then by Bochner's formula and  the formula \eqref{formula} for $|\partial h|^2$ and $|\bar\partial h|^2$ we obtain
\[\begin{split}\Delta J &=\Delta|\partial h|^2-\Delta |\bar\partial h|^2  \\&=|\partial h|^{-2} |\nabla |\partial h|^2|^2-|\bar\partial h|^{-2} | \nabla |\bar\partial h|^2|^2 -K_2 J_h D_h,\end{split}\] where $$D_h = |\partial h|^2+|\bar \partial h|^2.$$
Thus
$$\Delta J =   4|\nabla |\partial h||^2-  4| \nabla |\bar\partial h||^2 -K_2 J_h D_h.$$
Hence  \[\begin{split}J^2 \Delta \log J &= J \Delta J - |\nabla J|^2\\&=J( 4|\nabla |\partial h||^2- 4 | \nabla |\bar\partial h||^2 -K_2 J_h D_h)- |\nabla J|^2. \end{split}\]

On the sequel we use the formulas $$|h_z|^2=h_z \bar h_{\bar z}, \ \text{and}, \ \ |h_{\bar z}|^2=h_{\bar z}\bar h_z$$ and for a real function $R$ the formula $$|\nabla R|=2|R_z|.$$
Since 
$$J=|\partial h|^2-|\partial \bar h|^2=\rho^2(|h_z|^2-|h_{\bar z}|^2),$$  we obtain that 
$$|\nabla J|^2=4 |J_z|^2= 4|\left(\rho^2(|h_z|^2-|h_{\bar z}|^2)\right)_z|^2$$
which yields
$$|\nabla J|^2 = 4|\rho^2\left(h_{zz} \bar h_{\bar z}+h_z \bar h _{z \bar z}-h_{z\bar z} \bar h_z - h_{\bar z}\bar h_{zz}\right)+2\rho\rho_z(|h_z|^2-|h_{\bar z}|^2)|^2$$
and $$|\nabla |\partial h||^2=4||\partial h|_z|^2=4|\rho_z |h_z|+\rho|h_z|_z|^2=4\left|\rho_z |h_z|+\frac{\rho}{2|h_z|}(h_{zz}\bar h_{\bar z}+h_z\bar h_{z\bar z})\right|^2,$$

 $$|\nabla |\bar\partial h||^2=4||\bar\partial h|_{ z}|^2=4|\rho_z |h_{\bar z}|+\rho|h_{\bar z}|_z|^2=4\left|\rho_z |h_{\bar z}|+\frac{\rho}{2|h_{\bar z}|}(h_{\bar zz}\bar h_{ z}+h_{\bar z}\bar h_{z z})\right|^2$$
and 
$$( |\nabla |\partial h||^2-  | \nabla |\bar\partial h||^2) = \frac{|\nabla \rho|^2 J}{\rho^2} + {\rho^2}\left(\frac{|h_{zz} \bar h_{\bar z}+\bar h_{z\bar z} h_z|^2}{|h_z|^2}-\frac{|h_{z\bar z} \bar h_z+h_{\bar z} \bar h_{zz}|^2}{|h_{\bar z}|^2}\right) $$ $$ + 4\Re (\rho\rho_z(h_{zz} \bar h_{\bar z}+\bar h_{z\bar z} h_z-h_{z\bar z} \bar h_z-h_{\bar z} \bar h_{zz}) ).$$
Hence
\[ \begin{split}|\nabla J|^2-J \Delta J  &= 4|\rho^2\left(h_{zz} \bar h_{\bar z}+h_z \bar h _{z \bar z}-h_{z\bar z} \bar h_z - h_{\bar z}\bar h_{zz}\right)+2\rho\rho_z(|h_z|^2-|h_{\bar z}|^2)|^2\\&- 4J\bigg(\frac{|\nabla \rho|^2 J}{\rho^2} + {\rho^2}\left(\frac{|h_{zz} \bar h_{\bar z}+\bar h_{z\bar z} h_z|^2}{|h_z|^2}-\frac{|h_{z\bar z} \bar h_z+h_{\bar z} \bar h_{zz}|^2}{|h_{\bar z}|^2}\right) \\& + 4\Re (\rho\rho_{\bar z}(h_{zz} \bar h_{\bar z}+\bar h_{z\bar z} h_z-h_{z\bar z} \bar h_z-h_{\bar z} \bar h_{zz}) )\bigg)\\&  +K_2 J^2 D_h.\end{split}\]

Thus 

\[ \begin{split}|\nabla J|^2-J \Delta J  &= 4|\rho^2\left(A-B\right)+2\rho\rho_z(|h_z|^2-|h_{\bar z}|^2)|^2\\&- 4J\bigg(\frac{|\nabla \rho|^2 J}{\rho^2} + {\rho^2}\left(\frac{|A|^2}{|h_z|^2}-\frac{|B|^2}{|h_{\bar z}|^2}\right) + 4\Re (\rho\rho_{\bar z}(A-B)\bigg)\\&  +K_2 J^2 D_h.\end{split}\]

After some easy manipulations we arrive at the equation

\begin{equation}\label{eqw}|\nabla J|^2-J \Delta J = K_2 J^2 D_h + 4\rho^4\frac{|h_z|^2}{|h_{\bar z}|^2}|B|^2+4\rho^4\frac{|h_{\bar z}|^2}{|h_{ z}|^2}|A|^2-8\rho^4\Re  (A\bar B)\end{equation} where
\underline{}
$$A =h_{zz} \bar h_{\bar z}+\bar h_{z\bar z} h_z$$

and $$B=h_{z\bar z} \bar h_z+h_{\bar z} \bar h_{zz}. $$
From \eqref{eqw} and formula \eqref{formula} we obtain \eqref{del} for the points on the domain where $\partial h$ and $\bar \partial h$ does not vanish.

But Proposition~\ref{prop2} says that the isolated zeros of  $|\partial h|$ and $|\bar \partial h|$ are isolated, and thus \eqref{del} is valid except for a set of isolated points of $D$. 

It remains to observe that \[\begin{split}\frac{|h_z|^2}{|h_{\bar z}|^2}|B|^2&+\frac{|h_{\bar z}|^2}{|h_{ z}|^2}|A|^2-2\Re (A\bar B) \\&=
\left(\frac{|h_z|}{|h_{\bar z}|}|B|-\frac{|h_{\bar z}|}{|h_{ z}|}|A|\right)^2+2(|A| |B|-\Re(A\bar B))\ge 0.\end{split}\]
Since  $\log J_h$ is continuous, we conclude that $-\log J_u$ is subharmonic  under the constraint $K_2 \ge 0.$
\end{proof}
Since a harmonic homeomorphism has the non-vanishing Jacobian (see e.g. \cite{Schulz} or \cite{heinz}), we know that $\log J_u$ is well defined. By maximum principle for subharmonic functions  we obtain
\begin{corollary}
Assume that $h$ is a homeomorphism between surfaces $(D,\sigma)$ and $(S,\rho)$ and let $ \gamma\subset D$ be a  Jordan curve surrounding a domain $\Omega\subset D$. Assume further that the curvature of $S$ is non-negative. Then $$\min\{J_u(z): z\in \Omega\cup \gamma\}=\min_{z\in \gamma} J_u(z).$$
\end{corollary}
\begin{corollary}
Let $h$ be as in Theorem~\ref{theo} and let $J_h^0$ be the Euclidean Jacobian of $h$. Let $r(x,y)=|h(x,y)|$ and assume that $\rho$ is rotationally symmetric. Then \[\begin{split}-\Delta \log J^0&=K_2 \rho^2(D^0_h-|\nabla r|^2) +\frac{2\rho'}{\rho r}(r\Delta r-|\nabla r|^2)\\&+ \frac{4\rho^4}{J^2}\left(\frac{|h_z|^2}{|h_{\bar z}|^2}|B|^2+\frac{|h_{\bar z}|^2}{|h_{ z}|^2}|A|^2-2\Re (A\bar B)\right),\end{split}\] where $A$ and $B$ are as in  Theorem~\ref{theo}  and $D^0_h=|h_z|^2+|h_{\bar z}|^2$.

\end{corollary}
\begin{proof}
Since $J=\rho^2 J^0$, and $$-\Delta \log J=-\Delta \log \rho^2-\Delta \log J^0,$$ we obtain that $$-\Delta \log J^0=\Delta \log \rho^2-\Delta \log J^0.$$ So from \eqref{del} we obtain

For rotationally symmetric metric $\rho(z)=\rho(r)$, where $r=|z|$,  we have
$$K_2=-\frac{\Delta\log \rho(r)^2}{\rho(r)^2}=-\frac{2 \left(\rho(r) \rho''(r)- \rho'(r)^2+\rho(r) \rho'(r)/r\right)}{ \rho(r)^4}.$$

$$ \rho(r) \rho''(r)- \rho'(r)^2=-\rho(r) \rho'(r)/r+ \rho(r)^4 K_2/(-2)$$

Now if $r(x,y)=|h(x,y)|$ we obtain $${\Delta\log \rho(r)^2}=\frac{2\left((\rho \rho''-(\rho')^2)|\nabla r|^2+\rho \rho '\Delta r \right)}{\rho^2}.$$ Thus

$${\Delta\log \rho(r)^2}=\frac{\left((-\frac{2\rho \rho'}{r}-{K_2\rho^4})|\nabla r|^2+2\rho \rho '\Delta r \right)}{\rho^2}.$$
Hence
\[\begin{split}-\Delta \log J^0&=-\Delta \log J-K_2\rho^2\\&=K_2 \rho^2(D^0_h-|\nabla r|^2) +\frac{2\rho'}{\rho r}(r\Delta r-|\nabla r|^2)\\&+ \frac{4\rho^4}{J^2}\left(\frac{|h_z|^2}{|h_{\bar z}|^2}|B|^2+\frac{|h_{\bar z}|^2}{|h_{ z}|^2}|A|^2-2\Re (A\bar B)\right).\end{split}\]

\end{proof}

\end{document}